\documentclass[11pt]{amsart}
\usepackage[utf8x]{inputenc}
\usepackage{amsmath, amsthm, amssymb, amscd, comment, enumerate, textcomp, units,stmaryrd, datetime}
\usepackage{cite}
\usepackage{graphicx}
\usepackage[all, arrow]{xy}\SelectTips{cm}{10}
\usepackage{lipsum}
\makeatletter
\def\@setcopyright{}
\def\serieslogo@{}
\makeatother
\begin{document}

\baselineskip=14.5pt

\bibliographystyle{plain}

\makeatletter
\def\imod#1{\allowbreak\mkern10mu({\operator@font mod}\,\,#1)}
\makeatother
\newcommand{\vectornorm}[1]{\left|\left|#1\right|\right|}
\newcommand{\ip}[2]{\left\langle#1,#2\right\rangle}
\newcommand{\ideal}[1]{\left\langle#1\right\rangle}
\newcommand{\sep}[0]{^{\textup{sep}}}
\newcommand{\Span}[0]{\operatorname{Span}}
\newcommand{\Stab}[0]{\operatorname{Stab}}
\newcommand{\Orb}[0]{\operatorname{Orb}}
\newcommand{\Soc}[0]{\operatorname{Soc}}
\newcommand{\Kernel}[0]{\operatorname{ker }}
\newcommand{\Gal}[0]{\operatorname{Gal}}
\newcommand{\Aut}[0]{\operatorname{Aut}}
\newcommand{\Out}[0]{\operatorname{Out}}
\newcommand{\Inn}[0]{\operatorname{Inn}}
\newcommand{\Image}[0]{\operatorname{im }}
\newcommand{\vol}[0]{\operatorname{vol }}
\newcommand{\pgl}[0]{\operatorname{PGL}_2(\mathbb{F}_3)}
\newcommand{\AGL}[0]{\operatorname{AGL}}
\newcommand{\PGL}[0]{\operatorname{PGL}}
\newcommand{\PSL}[0]{\operatorname{PSL}}
\newcommand{\PSp}[0]{\operatorname{PSp}}
\newcommand{\PSU}[0]{\operatorname{PSU}}
\newcommand{\Sp}[0]{\operatorname{Sp}}
\newcommand{\POmega}[0]{\operatorname{P\Omega}}
\newcommand{\Uup}[0]{\operatorname{U}}
\newcommand{\Gup}[0]{\operatorname{G}}
\newcommand{\Fup}[0]{\operatorname{F}}
\newcommand{\Eup}[0]{\operatorname{E}}
\newcommand{\Bup}[0]{\operatorname{B}}
\newcommand{\Dup}[0]{\operatorname{D}}
\newcommand{\Mup}[0]{\operatorname{M}}
\newcommand{\mini}[0]{\operatorname{min}}
\newcommand{\maxi}[0]{\textup{max}}
\newcommand{\modu}[0]{\textup{mod}}
\newcommand{\nth}[0]{^\textup{th}}
\newcommand{\sd}[0]{\leq_{sd}}
 \newtheorem{theorem}{Theorem}
 \newtheorem{proposition}[theorem]{Proposition}
 \newtheorem{lemma}[theorem]{Lemma}
 \newtheorem{corollary}[theorem]{Corollary}
 \newtheorem{conjecture}[theorem]{Conjecture}
 \newtheorem{definition}[theorem]{Definition}
 \newtheorem{question}[theorem]{Question}
 \newtheorem*{claim*}{Claim}
 \newtheorem{claim}[theorem]{Claim}
 \newtheorem{example}[theorem]{Example}
 \newtheorem*{example*}{Example}
 
 \newcommand{\mc}{\mathcal}
 \newcommand{\mcA}{\mc{A}}
 \newcommand{\mcB}{\mc{B}}
 \newcommand{\mcC}{\mc{C}}
 \newcommand{\mcD}{\mc{D}}
 \newcommand{\mcE}{\mc{E}}
 \newcommand{\mcF}{\mc{F}}
 \newcommand{\mcG}{\mc{G}}
 \newcommand{\mcH}{\mc{H}}
 \newcommand{\mcI}{\mc{I}}
 \newcommand{\mcJ}{\mc{J}}
 \newcommand{\mcK}{\mc{K}}
 \newcommand{\mcL}{\mc{L}}
 \newcommand{\mcM}{\mc{M}}
 \newcommand{\mcN}{\mc{N}}
 \newcommand{\mcO}{\mc{O}}
 \newcommand{\mcP}{\mc{P}}
 \newcommand{\mcQ}{\mc{Q}}
 \newcommand{\mcR}{\mc{R}}
 \newcommand{\mcS}{\mc{S}} 
 \newcommand{\mcT}{\mc{T}}
 \newcommand{\mcU}{\mc{U}}
 \newcommand{\mcV}{\mc{V}}
 \newcommand{\mcW}{\mc{W}}
 \newcommand{\mcX}{\mc{X}}
 \newcommand{\mcY}{\mc{Y}}
 \newcommand{\mcZ}{\mc{Z}}
 \newcommand{\AAA}{\mathbb{A}}
 \newcommand{\CC}{\mathbb{C}}
 \newcommand{\FF}{\mathbb{F}}
 \newcommand{\NN}{\mathbb{N}}
 \newcommand{\PP}{\mathbb{P}}
 \newcommand{\QQ}{\mathbb{Q}}
 \newcommand{\QQd}{\QQ^{(d)}}
 \newcommand{\RR}{\mathbb{R}}
 \newcommand{\ZZ}{\mathbb{Z}}
 \newcommand{\oneto}[1]{\{1,\dots,#1\}}

 \renewcommand{\thefootnote}{\fnsymbol{footnote}}

\def\today{
  \ifcase\month\or
  January\or February\or March\or April\or May\or June\or
  July\or August\or September\or October\or November\or December\fi
  \space\number\day, \number\year}

\author[Itamar Gal]{Itamar Gal}
\address{Department of Mathematics, University of Texas at Austin, 2515 Speedway, Stop C1200\\Austin, TX 78712, U.S.A.}
\email{igal@math.utexas.edu}
\author[Robert Grizzard]{Robert Grizzard$^*$}
\address{Department of Mathematics, University of Texas at Austin, 2515 Speedway, Stop C1200, Austin, TX 78712, U.S.A.}
\email{rgrizzard@math.utexas.edu}
\thanks{$^*$The second author's research was partially supported by a grant from the National Security Agency, H98230-12-1-0254.}
\title[On the compositum]{On the compositum of all degree $d$ extensions of a number field}

\begin{abstract}
Let $k$ be a number field, and denote by $k^{[d]}$ the compositum of all degree $d$ extensions of $k$ in a fixed algebraic closure.  We first consider the question of whether all algebraic extensions of $k$ of degree less than $d$ lie in $k^{[d]}$.  We show that this occurs if and only if $d < 5$.  Secondly, we consider the question of whether there exists a constant $c$ such that if $K/k$ is a finite subextension of $k^{[d]}$, then $K$ is generated over $k$ by elements of degree at most $c$.  This was previously considered by Checcoli.  We show that such a constant exists if and only if $d < 3$.  This question becomes more interesting when one restricts attention to Galois extensions $K/k$.  In this setting, we 
derive certain divisibility conditions on $d$ under which such a constant does not exist.  If $d$ is prime, we prove that all finite Galois subextensions of $k^{[d]}$ are generated over $k$ by elements of degree at most $d$.
\end{abstract}

\maketitle

\section{Introduction}\label{intro}
Let $k$ be a field.  Throughout this paper, all extensions of $k$ will be assumed to lie in a fixed algebraic closure $\overline{k}$.  We are interested in fields obtained by adjoining to $k$ all roots of irreducible polynomials of a given degree $d$.  For any positive integer $d$ we will write
\begin{align}
k^{[d]} &= k(\beta ~\big|~ [k(\beta) : k] = d), ~\textup{and}\label{kbrd}\\
k^{(d)} &= k(\beta ~\big|~ [k(\beta) : k] \leq d) = k^{[2]}k^{[3]}k^{[4]} \cdots k^{[d]}.\label{kpad} 
\end{align}

We have $k^{[1]} = k^{(1)} = k$, and when $d > 1$ it is clear that $k^{[d]}$ and $k^{(d)}$ are normal extensions of $k$.  We are primarily interested in the case where $k$ is a number field, in which case these are infinite Galois extensions.  When $d > 2$ it is natural to ask what polynomials of degree less than $d$ split in $k^{[d]}$.  If $c < d$ and all irreducible polynomials of degree $c$ split in $k^{[d]}$, then $k^{[c]} \subseteq k^{[d]}$.  Notice that this occurs in particular when $c$ divides $d$, since every degree $c$ extension admits a degree $d/c$ extension.  If all polynomials of degree less than $d$ split in $k^{[d]}$, then $k^{[d]} = k^{(d)}$.  We will prove the following results along these lines.

\begin{theorem}\label{thm1} If $k$ is a number field \footnote[2]{Many of our results contain the hypothesis that $k$ is a number field or global function field.  However, the astute reader will notice after reading the proofs that this hypothesis could be replaced with more technical restrictions on the field $k$ -- specifically, that certain embedding problems have solutions over $k$.}, then 
\begin{enumerate}
 \item[(a)] $k^{[2]} \subseteq k^{[d]}$ for all $d \geq 2,$
 \item[(b)] $k^{[3]} \subseteq k^{[4]},$ and
 \item[(c)] for each $d \geq 5$, there exists a prime $p < d$ such that $k^{[p]} \not \subseteq k^{[d]}.$
\end{enumerate}
\end{theorem}

The following corollary is immediate.

\begin{corollary}\label{kdkd} If $k$ is a number field, then $k^{[d]} = k^{(d)}$ if and only if $d < 5$.
\end{corollary}

We now introduce the notion of boundedness for an extension of fields.  We will use this language to state our remaining results.

\begin{definition}\label{boundeddef}
We say an infinite extension $M$ of $k$ is \emph{bounded over $k$} (or that $M/k$ is \emph{bounded}) if there exists a constant $c$ such that all finite subextensions of $M/k$ can be generated by elements of degree less than or equal to $c$.  
If there is no such $c$, we say that $M/k$ is \emph{unbounded}.  

If all finite Galois subextensions of $M/k$ can be generated by elements of degree less than or equal to $c$, we say $M/k$ is \emph{Galois bounded}; otherwise we say $M/k$ is \emph{Galois unbounded}.
\end{definition}
It was first shown by Checcoli that, for a number field $k$, the extension $k^{(d)}/k$ is not in general Galois bounded (see \cite{checcolithesis}, Theorem 2, part ii).  We will address the question of how boundedness and Galois boundedness depend on $d$ for the fields $k^{(d)}$ and $k^{[d]}$.  Further restricting attention to abelian Galois extensions greatly simplifies the discussion.  It is easily seen that $k^{(d)}_{\textup{ab}}$ is bounded over $k$ for all $d$, where the subscript denotes the maximal abelian subextension (cf. \cite{checcolithesis} Theorem 1.4).

In the case where $k$ is a number field,  Bombieri and Zannier ask in \cite{bombierizannier} whether, for any given constant $T$, only finitely many points in $k^{(d)}$ have absolute Weil height (see \cite{bombierigubler}, p. 16 for a definition) at most $T$.  Such a finiteness property is called the \emph{Northcott property}.  This problem has been further discussed in \cite{widmerN} and \cite{checcoliwidmer}, but remains open.  In Theorem 1 of \cite{bombierizannier} it is proved that this property is enjoyed by $k^{(d)}_{\textup{ab}}$, and the boundedness of $k^{(d)}_{\textup{ab}}/k$ plays a role in the proof.  The authors of the present work are hopeful that understanding the boundedness properties in $k^{[d]}$ and $k^{(d)}$ will be useful in understanding such problems.

The following theorems summarize our results on boundedness and Galois boundedness.

\begin{theorem}\label{bdd}
If $k$ is a number field, then $k^{[d]}$ is bounded over $k$ if and only if $d \leq 2$.
\end{theorem}
\begin{theorem}\label{galbdd}
If $k$ is any field and $p$ is a prime, then $k^{[p]}$ is Galois bounded over $k$.  More precisely, all finite subextensions of $k^{[p]}/k$ can be generated by elements of degree at most $p$ over $k$.
\end{theorem}
We will also establish the following partial converse to Theorem \ref{galbdd}.
\begin{theorem}\label{squareeven}
If $k$ is a number field or global function field and $d > 2$, then $k^{[d]}/k$ is Galois unbounded in the following cases:
\begin{enumerate}
\item[(a)] $d$ is divisible by a square;
\item[(b)] $d$ is divisible by two primes $p$ and $q$ such that $q  \equiv 1 \imod{p}$.  
\end{enumerate}
In particular, this includes the case where $d$ is even and greater than 2.
\end{theorem}

In terms of the fields $k^{(d)}$, Theorems \ref{bdd}, \ref{galbdd}, and \ref{squareeven} immediately imply the following.
\begin{corollary}\label{paren}
Let $k$ be a number field.  Then 
\begin{enumerate}
\item[(a)] $k^{(2)}/k$ is bounded, 
\item[(b)] $k^{(3)}/k$ is Galois bounded but not bounded, and 
\item[(c)] $k^{(d)}/k$ is Galois unbounded for $d \geq 4$.
\end{enumerate}
\end{corollary}

This paper is organized as follows.  Sections 2 and 3 are devoted to preliminaries and background material on group theory and Galois theory.  In Section 4 we prove Theorem \ref{thm1}; parts (a) and (b) appeal to existing results on embedding problems, while part (c) follows by a purely group theoretic argument.  We conclude Section 4 with an elementary construction which gives part (a) in the case where $k = \QQ$.  In Section 5 we prove Theorems \ref{bdd} and \ref{squareeven} using explicit constructions.  Finally, in Section 6 we prove Theorem \ref{galbdd} as an immediate corollary of a purely group theoretic statement (see Proposition \ref{bigboy}).

 \section*{Acknowledgments}
 The authors would like to thank Daniel Allcock, Sara Checcoli, Joseph Gunther, Andrea Lucchini, Jeffrey Vaaler, and anonymous referees for numerous useful communications.  We would also like to express our appreciation to the \verb|GAP| group.  Although we did not use any computer calculations directly for any of the results in this paper, we used the \verb|GAP| software package extensively to improve our understanding of the group theoretic aspects of these problems.


\section{Preliminaries on group theory}\label{groupsec}

We recall some standard definitions.  A \emph{transitive group} of degree $d$ will mean a finite permutation group acting faithfully and transitively on a set $\Omega$ of size $d$, such as the Galois group of an irreducible degree $d$ polynomial acting on the roots.  A transtive group is \emph{primitive} if there is no nontrivial partition of $\Omega$ such that the group has an induced action on the blocks of the partition.  Since all such blocks must be equal in size, this implies that any transitive group of prime degree is primitive.  For more background on transitive and primitive groups, see \cite{dixonmortimer} or \cite{wielandt}. 

Let us fix some notation for finite groups.  We will denote by $C_d$, $D_d$, $A_d$, and $S_d$ the cyclic, dihedral, alternating, and symmetric groups of degree $d$, respectively.  Note that $D_d$ has order $2d$.  We denote the Klein 4-group by $V$.

A \emph{subdirect product} $G$ of some collection of groups $\{G_i\}_i$ is a subgroup of the direct product $\prod_iG_i$ with the property that the projection map from $G$ to each factor $G_i$ is surjective.

 Let $H_1,H_2$ and $Q$ be groups, and let $\alpha_1 : H_1 \to Q$ and $\alpha_2 : H_2 \to Q$ be surjective group homomorphisms.  The \emph{fibered product} of $H_1$ with $H_2$ over $Q$ (with respect to the maps $\alpha_1$ and $\alpha_2$) is defined to be the subgroup $H_1\times_QH_2$ of the direct product $H_1\times H_2$ given by
$$H_1\times_QH_2 = \{ (h_1,h_2) \in H_1\times H_2 \hspace{1pt} ~\big|~ \hspace{1pt} \alpha_1(h_1) = \alpha_2(h_2) \}.$$  Notice that we have
\begin{equation}\label{fibord}
 |H_1 \times_Q H_2| = \frac{|H_1| \cdot |H_2|}{|Q|}.
\end{equation}

The following lemma can be found in different forms in many texts, and is variously attributed to Goursat or Goursat and Lambek.  A short proof can be found in \cite{butlermckay}, p. 864.

\begin{lemma}[Goursat's Lemma]\label{gl} Let $H_1$ and $H_2$ be groups. The set of subdirect products of $H_1\times H_2$ is equal to the set of fibered products $H_1\times_Q H_2$. In particular, every subdirect product of $H_1\times H_2$ is of the form $H_1\times_QH_2$.
\end{lemma}


\section{Galois theory and embedding problems}\label{embeddingsec}
The following elementary proposition highlights the role of Galois theory in the proofs of our results.
\begin{proposition}\label{rirw}
Let $k$ be a perfect field and let $L/k$ be a finite Galois extension of fields.  The following are equivalent:
\begin{enumerate}
 \item[(a)] $L$ is generated by elements of degree $d$ over $k$;
 \item[(b)] in $\Gal(L/k)$ the trivial group is the intersection of subgroups of index $d$;
 \item[(c)] $\Gal(L/k)$ is a subdirect product of transitive groups of degree $d$.
\end{enumerate}
\end{proposition}
\begin{proof}
The equivalence (a) and (b) follows immediately from the Galois correspondence and the primitive element theorem.  If (a) is satisfied, then $L$ is a compositum of the splitting fields of some degree $d$ polynomials.  It follows from basic Galois theory that $\Gal(L/k)$ is a subdirect product of these Galois groups, which are transitive groups of degree $d$, so (c) is satisfied.  Suppose (c) is satisfied, so we have $\Gal(L/k)$ acting on a disjoint union of sets of size $d$, transitively on each set.  Then all point-stabilizers have index $d$, and the intersection of these subgroups is trivial, yielding (b).
\end{proof}

In order to establish Theorem \ref{thm1}, we must discuss the embedding problem in Galois theory.  Let $K/k$ be a Galois extension of fields, $G$ a finite group, and $N$ a normal subgroup of $G$ with a short exact sequence 
 \begin{equation}\label{embseq}
 1 \to N \to G \overset{\phi}{\to} \Gal(K/k) \to 1.
\end{equation} 

This data gives us the \emph{embedding problem} $(K/k, G, N)$.  A \emph{solution} to the embedding problem is an extension $L/k$ with $L \supseteq K$ such that $\Gal(L/k) \cong G$ and the natural map $\Gal(L/k) \to \Gal(K/k)$ agrees with $\phi$.  Hence, a solution to the embedding problem is described by the following commutative diagram.

\begin{equation}\label{solution}
\begin{xy}
(22,8)*+{\Gal(L/k)}="f";%
(0,-8)*+{1}="a"; (10,-8)*+{N}="b"; (22,-8)*+{G}="c"; (40,-8)*+{\Gal(K/k)}="d"; (57,-8)*+{1.}="e";%
{\ar "a";"b"}; {\ar "b";"c"}; {\ar^{\phi\hspace{20pt}} "c";"d"}; {\ar "d";"e"};%
{\ar^{\wr} "f";"c"}; {\ar "f";"d"};%
\end{xy}
\end{equation}

For our purposes, all that is important is finding an extension $L/k$ such that $L \supseteq K$ and $\Gal(L/k) \cong G$, and therefore we will not mention the map $\phi$ in what follows.  

A celebrated result in this context is a theorem of Shafarevich, which states that if $k$ any number field or global function field, any solvable group can be realized as the Galois group of some extension of $k$.  Since products of solvable groups are solvable, this allows us to realize a solvable group as the Galois group of infinitely many extensions, whose pairwise intersections are $k$.  A full proof of Shafarevich's Theorem, along with more background on embedding theory, can be found in \cite{cnf}.

The following proposition is a simple yet important observation which is used implicitly throughout the proof of Theorem \ref{thm1}.
\begin{proposition}\label{newprop}
Let $k$ be a field and let $K/k$ be a finite extension.  Then $K \subseteq k^{[d]}$ if and only if the following two conditions are met.
\begin{enumerate}
 \item[(i)] We can find a group $H$ which is a subdirect product of transitive groups of degree $d$ with some normal subgroup $N$ such that there is a short exact sequence 
\begin{equation}\label{ses2} 
1 \to N \to H \to \Gal(K/k) \to 1.
\end{equation} 
\item[(ii)] We can solve the corresponding embedding problem, i.e. find $L \supseteq K$ such that $\Gal(L/k) \cong H$.
\end{enumerate}
\end{proposition}
\begin{proof}
If $K \subseteq k^{[d]}$, then $K$ is contained in some finite Galois extension $L/k$ generated by elements of degree $d$.  By Proposition \ref{rirw}, we have that $\Gal(L/k)$ is a subdirect product of transitive groups of degree $d$, and (i) and (ii) are clearly satisfied via the short exact sequence 
\begin{equation}\label{ses1}
1 \to \Gal(L/K) \to \Gal(L/k) \to \Gal(K/k) \to 1.
\end{equation}

Conversely, if (i) and (ii) are satisfied, then we have $K \subseteq L$ as in (ii), and $L \subseteq k^{[d]}$ by (i) and Proposition \ref{rirw}.
\end{proof}

\section{Proof of Theorem \ref{thm1}} \label{pf}
We implicitly apply Proposition \ref{newprop} throughout.  For integers $m<d$, we are interested in whether or not $k^{[m]} \subseteq k^{[d]}$.  Let $K$ be the splitting field of an irreducible polynomial of degree $m$ in $k[x]$.  In the case $m = 2$, we must have that $\Gal(K/k) \cong C_2$, and we use the following result due to O. Neumann (cf. \cite{neumann}, Theorem 2) in order to conclude that $K \subseteq k^{[d]}$.
\begin{proposition}\label{23}
Let $K/k$ be a quadratic extension of number fields and let $d \geq 3$.  Then there is a solution to the embedding problem $(K/k, S_d, A_d)$ arising from 
\begin{equation}\label{sesas}
 1 \to A_d \to S_d \to \Gal(K/k) \to 1.
\end{equation}
In other words, every irreducible quadratic splits in the splitting field of some degree $d$ polynomial (with symmetric Galois group).
\end{proposition}
This establishes part (a) of Theorem \ref{thm1}, that $k^{[d]} \subseteq k^{[d]}$ for all $d \geq 2$, and in particular it tells us that $k^{[3]} = k^{(3)}$.   At the end of this section we give a short, elementary proof of part (a) of Theorem 1 in the case where $k = \QQ$.

For part (b) of Theorem 1 we consider the case $m = 3, d = 4$.  We must have $\Gal(K/k) \cong S_3$ or $C_3$.  The following is a special case of a classical result of Shafarevich that gives the solution to all embedding problems with nilpotent kernel (see \cite{serregt}, Claim 2.2.5).
\begin{proposition}\label{34}
Let $k$ be a number field and let $f(x) \in k[x]$ be an irreducible cubic with splitting field $K$.  Let $V$ denote the Klein 4-group.
\begin{enumerate}
 \item[(a)] If $\Gal(K/k) \cong S_3$, then there is a solution to the embedding problem $(K/k, S_4, V)$ arising from $$1 \to V \to S_4 \to \Gal(K/k) \to 1.$$
 \item[(b)] If $\Gal(K/k) \cong C_3$, then there is a solution to the embedding problem $(K/k, A_4, V)$ arising from $$1 \to V \to A_4 \to \Gal(K/k) \to 1.$$
\end{enumerate}
In other words, every irreducible cubic splits in the splitting field of some quartic.
\end{proposition}
This proves that $k^{[3]} \subseteq k^{[4]}$, and combining with part (a) of Theorem 1 we now have that $k^{[4]}=k^{(4)}$.

To prove part (c) of Theorem \ref{thm1} we consider the case $d \geq 5$.  We will show that, for certain primes $p < d$, if $\Gal(K/k) \cong C_p$, then there is no possible subdirect product of transitive groups of degree $d$ having $\Gal(K/k)$ as a quotient. That is, we cannot even find groups $H$ and $N$ satisfying a short exact sequence as in (\ref{ses2}) above.  We begin with a lemma.
\begin{lemma}\label{pcycleAn}
For any integer $d \geq 5$ there exists a prime number $p \in (\frac{d}{2},d)$ such that, if $G$ is a transitive subgroup of $S_d$ containing a $p$-cycle, then either $G = S_d$ or $G = A_d$. 
\end{lemma}

\begin{proof}
The transitive groups of degree $d$ are well-known for small $d$ -- see for example \cite{butlermckay} for the groups up to degree 11; GAP (see \cite{GAP4}, \cite{GAPtr}) has a library of all of them for $d \leq 30$).  It can be checked easily that we can use $p = 3$ when $d=5$, and we can use $p=5$ when $d=6,7$; in each of these cases, $S_d$ and $A_d$ are the only transitive subgroups with order divisible by $p$.  Therefore all that remains is to prove our lemma in the case $d \geq$ 8.

There exists at least one prime $p \in (\frac{d}{2},d-2)$.  This follows from Bertrand's Postulate, first proved by Chebyshev, which states that for $m > 3$ there exists a prime in the interval $(m, 2m-2)$ -- see \cite{hardywright}, p. 343, Theorem 418; cf. p. 373.  Let $p$ be such a prime, and suppose $G$ be a transitive subgroup of $S_d$ containing some $p$-cycle $g$.  Without loss of generality, $g = (1~2~3~\cdots~p)$.  Since $G$ is transitive, for each $i \in \{p+1,\dots,d\}$ there is some element $\sigma_i \in G$ such that $\sigma_i(1) = i$.  If we let $g_i = \sigma_i g \sigma_i^{-1}$, then $g_i$ will be a $p$-cycle in $G$ whose support contains $i$.  Since $p$ is prime, each $\langle g_i \rangle$ acts primitively on its support, which is a set of size $p$.  Since $p > \frac{d}{2}$, the pairwise intersections of the supports of the groups $\langle g_i \rangle$ are nontrivial.  Therefore we can apply Proposition 8.5 from 
\cite{wielandt} inductively to see that the subgroup $H = \langle g, g_{p+1},g_{p+2},\dots,g_{d} \rangle$ is a primitive subgroup of $S_d$.  Since $H$ contains a $p$-cycle and $p < d-2$, Theorem 13.9 from \cite{wielandt} tells us that either $H = S_d$ or $H = A_d$, and since $H \leq G$, our proof is complete.
\end{proof}

Part (c) will be an immediate corollary of the following proposition.
\begin{proposition}\label{noquot}
For any integer $d \geq 5$ there exists some prime $p < d$ such that, if $G \leq_{sd} G_1 \times \cdots \times G_n$ is a subdirect product of transitive groups of degree $d$, then the cyclic group $C_p$ is not a quotient of $G$.
\end{proposition}

\begin{proof}
Fix $d \geq 5$.  By Lemma \ref{pcycleAn}, there is a prime $p \in (\frac{d}{2},d)$ such that the only transitive subgroups of $S_d$ containing a $p$-cycle are $S_d$ and $A_d$.  We proceed by induction on $n$, noting that the case $n=1$ follows immediately by our choice of $p$.  In general, we will have that $G \leq_{sd} G_0 \times G_n$, where $G_n$ is a transitive group of degree $d$ and $G_0$ is a subdirect product of $n-1$ such groups.  If $N$ is any normal subgroup of $G$, we have that $N \leq_{sd} N_0 \times N_n$ for some normal subgroups $N_0 \unlhd G_0$ and $N_n \unlhd G_n$.   By Goursat's Lemma, we may write $G$ as a fibered product $G = G_0 \times_Q G_n$ for some group $Q$ which is a quotient of both $G_0$ and $G_n$.  Similarly, we have $N = N_0 \times_R N_n$ for some group $R$ which is a quotient of both $N_0$ and $N_n$. 

By the inductive hypothesis, neither $G_0/N_0$ nor $G_n/N_n$ has order $p$.  Suppose that $G/N \cong C_p$.  Since $G/N$ surjects onto both $G_0/N_0$ and $G_n/N_n$, the latter two groups must be trivial.  Therefore, using (\ref{fibord}), we have
\begin{equation}\label{qr}
p = \frac{|G|}{|N|} = \frac{|G_0|\cdot|G_n|/|Q|}{|N_0|\cdot|N_n|/|R|} = |G_0/N_0|\cdot |G_n/N_n| \cdot \frac{|R|}{|Q|} = \frac{|R|}{|Q|}.
\end{equation}
This means that $|R|$ is divisible by $p$, and therefore $|G_n|$ and $|N_n|$ are both divisible by $p$ as well.  This means $G_n$ must be isomorphic to either $S_d$ or $A_d$.  Hence the only possibilities for $Q$ are $S_d$, $A_d$, $C_2$, or 1, and the only possibilities for $R$ are $S_d$ or $A_d$.  None of these possibilities allows for the equality in (\ref{qr}).
\end{proof}

This establishes part (c) of Theorem \ref{thm1}.  
In summary, if $d \leq 4$, an irreducible polynomial in $k[x]$ of degree less than $d$ splits in the splitting field of a single irreducible polynomial of degree $d$.  When $d > 4$, however, some irreducible polynomials of degree less than $d$ do not split \emph{in any compositum of such splitting fields}.  We conclude this section by demonstrating that part (a) of Theorem \ref{thm1} can be proved by a very elementary construction when $k = \QQ$.

\begin{proof}[Elementary proof that $\QQ^{[2]} \subseteq \QQ^{[d]}$ for all $d \geq2$]
In general, $k^{[\ell]} \subseteq k^{[d]}$ if $\ell | d$.  Hence it will suffice to show that $\sqrt{p} \in \QQ^{[\ell]}$ for any prime $\ell \geq 3$, whenever $p$ is a rational prime or $p = -1$.  If $p$ is any rational prime or equal to $\pm 1$, define
\begin{equation}\label{fp}
f_p(x) = x^\ell - \ell(\ell p + 1)x + (\ell-1)(\ell p+1)
\end{equation}
The discriminant $\Delta_p$ of this polynomial is given by the following (see for example \cite{masserdisc}):
\begin{equation}\label{deltap}
(-1)^{(\ell-1)(\ell-2)/2}\Delta_p = -(\ell-1)^{\ell-1}\ell^{\ell+1}(\ell p +1)^{\ell-1} \cdot p.
\end{equation}
In particular, it follows that $\sqrt{p}$ will be in the splitting field of either $f_p(x)$ or $f_{-p}(x)$. We now show that $f_p(x)$ is irreducible. First notice that if $\ell\neq p$ then $f_p(x+1)$ is Eisenstein at $\ell$. Next we consider the case where $\ell=p$. To handle this case we use Dumas's irreducibility criterion, which relies on the Newton diagram of a polynomial. For background on Dumas's criterion and Newton diagrams, see section 2.2.1 of \cite{2009polynomials}. In particular, we use the corollary to Theorem 2.2.1 found on page 55; we state it here for convenience.
\begin{proposition}[Dumas's Irreducibilty Criterion]\label{dumas}
Let $f(x)\in\mathbb{Z}[x]$ be a polynomial. If, for a prime $p$, the Newton diagram for $f$ consists of precisely one segment; i.e. consists of a segment containing no points with integer coordinates, then $f$ is irreducible.
\end{proposition}
Applying Dumas's criterion in the case $l=p$, we find that a sufficient condition for the irreducibility of $f_p$ is the existence of a prime $q$ and an integer $m$ such that $q^m$ exactly divides $p^2+1$, such that $q$ is coprime to $p-1$, and and such that $m$ is coprime to $p$. Notice that
\begin{equation}
(p^2+1) - (p+1)(p-1) = 2.
\end{equation}
Since $2$ is an integer combination of $p^2+1$ and $p-1$, it follows that $\gcd(p^2+1,p-1)$ divides $2$. Also notice that
\begin{equation}
 p^2+1 = (p-1)^2 + 2(p-1) + 2 \equiv 2 \imod{4}.
\end{equation}
Thus $p^2+1$ is not a power of $2$, and we can take $q$ to be any one of its odd prime factors. Now choose $m$ such that $q^m$ exactly divides $p^2+1$. Since $p^2+1 < q^p$ for $p,q\geq 3$, it follows that $1<m<p$. Thus $m$ is coprime to $p$, which completes the proof.
\end{proof}

\section{Unboundedness: proofs of Theorems \ref{bdd} and \ref{squareeven}}\label{unsec}
In the spirit of Propisition \ref{rirw}, let $G$ be a finite group and $d$ a positive integer.  Suppose that $H$ is a subgroup of $G$ that cannot be written as an intersection of subgroups of index less than or equal to $d$ in $G$.  If $G$ is the Galois group of a field extension $L/k$, this implies that the fixed field $K$ of $H$ is not generated over $k$ by elements of degree less than or equal to $d$.  In order to prove unboundedness results, we must exhibit groups with these properties which can be realized as Galois groups of subextensions of $k^{[d]}$.
The example in the next lemma 
will be applied toward establishing Theorem \ref{bdd}. 

\begin{lemma}\label{dpcp}
Let $p$ be an odd prime number, and let 
\begin{equation}\label{dpcp1}
 G = D_p^{n-1} \times C_p = \langle r_1, s_1, \dots, r_{n-1}, s_{n-1}, r_n \rangle
\end{equation}
be the direct product of $n-1$ copies of the dihedral group $D_p$ and a cyclic group of order $p$, where for $i \in \oneto{n-1}$, the $i^{th}$ $D_p = \langle r_i, s_i \rangle$ is generated by the $p$-cycle $r_i$ and the 2-cycle $s_i$, and $C_p = \langle r_n \rangle$.  Let 
\begin{equation}\label{dpcp2}
H = \langle r_1 r_n, r_2 r_n, \dots, r_{n-1} r_n \rangle \leq G.  
\end{equation}
If $B$ is a subgroup of $G$ with $H \lneq B \leq G$, then $r_n \in B$.  In particular, the intersection of all such subgroups $B$ strictly contains $H$.
\end{lemma}
\begin{proof}
Let $G_p = \langle r_1, \dots, r_n \rangle$ be the unique Sylow $p$-subgroup of $G$, considered as an $n$-dimensional $\FF_p$-vector space.  Any Sylow 2-subgroup $G_2$ of $G$ will be an $(n-1)$-dimensional $\FF_2$-vector space which acts by conjugation on $G_p$, so that $G = G_p \rtimes G_2$.

Let $H \lneq B \leq G$.  Note that $H$ is a codimension 1 subspace of $G_p$,  so if $B$ contains any element of order $p$ not in $H$, then $B$ contains all of $G_p$.  If $B$ contains any involution $\tau \in G$, notice that there will be some $i$ such that $\tau$ acts non-trivially on the $i^{th}$ copy of $D_p$, so that $\langle r_i r_n, \tau \rangle$ will contain $r_n$.  Since every nontrivial element of $G$ is either of order $p$, an involution, or of order $2p$ (a power of which is an involution), this completes our proof.
\end{proof}

\begin{corollary}\label{punbounded}
Let $k$ be a number field or a global function field, and let $p$ be an odd prime number.  Then $k^{[p]}/k$ is unbounded.   
\end{corollary}
\begin{proof}
Let $G$ and $H$ be as in Lemma \ref{dpcp}.  Since $G$ is solvable we have an extension $L/k$ with $\Gal(L/k) \cong G$.  Let $L^H$ be the fixed field in $L$ of $H$, and notice that $L^H \subseteq k^{[p]}$.  It is clear from our construction that $[L^H:k] = p \cdot 2^{n-1}$.  The Galois correspondence tells us that every proper subextension of $L^H/k$ corresponds to a subgroup $B$ of $G$ with $H \lneq B \leq G$.  Furthermore, since the intersection of all such groups strictly contains $H$, the compositum of all proper subextensions of $L^H/k$ is strictly a subfield of $L^H$.  This shows that $L^H$ is not generated by elements of degree less than $p \cdot 2^{n-1}$.
\end{proof}

Notice that the field extension $L^H/k$ in the proof above is \emph{not} Galois ($H$ is not normal in $G$).  As we will prove in the next section, this was necessarily so.

In order to prove our Galois unboundedness results, we must now introduce extraspecial $p$-groups.  We write $H_p$ for the finite Heisenberg group of order $p^3$, when $p$ is a prime.  This group is defined as the multiplicative group of upper triangular matrices of the form
 \[ \left( \begin{array}{ccc}
1 & a & c \\
0 & 1 & b \\
0 & 0 & 1 \end{array} \right),\]
with $a$, $b$, and $c$ belonging to the finite field $\FF_p$. 

The group $H_p$ plays an important role in our Galois unboundedness results.  We review some of its properties.  First, $H_p$ has a natural action on the three-dimensional vector space $\FF_p^3$.  Analyzing this action, it is easy to see that when an element of $H_p$ acts on a vector, the third coordinate is fixed, and $H_p$ acts faithfully and transitively on a 2-dimensional affine subspace (the subspace with third coordinate equal to 1, say), which has $p^2$ elements.  Thus we see that $H_p$ is isomorphic to a transitive group of degree $p^2$. 

The group $H_p$ is an \emph{extraspecial $p$-group}, meaning its center, commutator, and Frattini subgroups coincide and have order $p$.  For our purposes the only relevant fact is that the commutator subgroup is cyclic of order $p$.  We can construct larger extraspecial $p$-groups as follows.  Let $n$ be a positive integer, and consider the normal subgroup $N_{p,n}$ of the direct product $H_p^n$ given by
\begin{equation} \label{centralguy} 
N_{p,n} = \{ (z_1^{a_1}, \dots , z_n^{a_n}) ~\big|~ \Sigma_{i=1}^n a_i \equiv 0 \imod{p} \}, 
\end{equation}
where $z_i$ generates the center of the $i^{th}$ copy of $H_p$.  The quotient $H_p^n/N$ is an extraspecial $p$-group of order $p^{2n+1}$ and exponent $p$, which we will denote by $E_{p,n}$.  

The following lemma can be found in \cite{checcolithesis} (cf. Proposition 2.4), and follows from some basic facts about the extraspecial $p$-groups (see for example \cite{berkovichv1} and \cite{doerkhawkes}).
\begin{lemma}\label{hh}
The intersection of all subgroups of index less than $p^n$ in $E_{p,n}$ contains the commutator subgroup.  In particular, this intersection is nontrivial.
\end{lemma}

Checcoli used this fact in \cite{checcolithesis} to show that, for a number field $k$, the extension $k^{(d)}/k$ is not in general Galois bounded.  The idea of using these groups is attributed to A. Lucchini.  However, the author was not concerned with the question of which values of $d$ suffered from this pathology, nor with the more general question of the boundedness of $k^{[d]}/k$.  The use of extraspecial $p$-groups (which are certainly not the only groups with properties like the conclusion of Lemma \ref{hh}, but are natural and easy to work with) remains our primary tool for proving that extensions are Galois unbounded.  The following lemma simplifies our application of this principle.

\begin{lemma}\label{heisenberger}
Let $d$ be a positive integer.  Suppose there is a prime number $p$ such that there is a solvable group $G$ which is a subdirect product of transitive groups of degree $d$, and a quotient of $G$ is isomorphic to $H_p$.  Then $k^{[d]}/k$ is Galois unbounded for any number field or global function field $k$.
\end{lemma}
\begin{proof}
By Shafarevich's Theorem, for any positive integer $n$ we can realize $G^n$ as the Galois group of some extension $L/k$, and we will have $L \subseteq k^{[d]}$.  There will be Galois subextension $K/k$ with Galois group $H_p^n$, and the subfield of $K$ corresponding to the normal subgroup defined in (\ref{centralguy}) will have Galois group $E_{p,n}$, and will therefore not be generated by elements of degree less than $p^n$.
\end{proof}

The following lemma gives a construction of a permutation group that will allow us to apply Lemma \ref{heisenberger} in our proof of part (b) of Theorem \ref{squareeven}.
\begin{lemma}\label{hberger}
Let $d = pq$, where $p$ and $q$ are primes with $q \equiv 1 \imod{p}$.  Then there exists a transitive group of degree $d$ which is isomorphic to $C_q^p\rtimes H_p$.
\end{lemma}
\begin{proof}
Write $q = mp+1$.  Consider $p$ sets $\Omega_i$ of size $q$, written $\Omega_i = \{(1_i,2_i,\dots,q_i\}$ for $i \in \FF_p$.  We write $\Omega$ for the disjoint union of the sets $\Omega_i$.  We will construct a group $G$ of permutations of $\Omega$, which acts imprimitively with respect to the partition into the sets $\Omega_i$.  Let $\sigma$ be the permutation $(1~2~\cdots~q)$.  Since $q \equiv 1 \imod{p}$, the $q$-cycle $\sigma$ is normalized by some $(q-1)$-cycle $\eta$ in the symmetric group $S_q$, and we let $\tau = \eta^m$, which is a disjoint product of $m$ $p$-cycles.  The permutations $\sigma$ and $\tau$ induce permutations on each set $\Omega_i$, which we denote by $\sigma_i$ and $\tau_i$.

We define $\alpha = \tau_0 \tau_1 \cdots \tau_{p-1}$, $\beta = \tau_0^0 \tau_1^1 \cdots \tau_{p-1}^{p-1}$, and define $\gamma$ to be the permutation on $\Omega$ sending $j_i$ to $j_{i+1}$.  Let $A = \langle \sigma_0, \sigma_1, \cdots \sigma_{p-1} \rangle \cong C_q^p$, and $B = \langle \alpha, \beta, \gamma \rangle.$  The interested reader will verify that $A$ is normalized by $B$, and that $B \cong H_p$ via
\begin{equation}\label{matrixiso}
\alpha \mapsto  
 \left( \begin{array}{ccc}
 1 & 0 & 1 \\
 0 & 1 & 0 \\
 0 & 0 & 1 \end{array} \right),\quad
 \beta \mapsto
  \left( \begin{array}{ccc}
 1 & 1 & 0 \\
 0 & 1 & 0 \\
 0 & 0 & 1 \end{array} \right), \quad
 \gamma \mapsto
  \left( \begin{array}{ccc}
 1 & 0 & 0 \\
 0 & 1 & 1 \\
 0 & 0 & 1 \end{array} \right).
\end{equation}
The example below with $p = 3$, $q=7$ makes the isomorphism more clear.  The Heisenberg group $B$ acts simultaneously on $m$ ``planes'' of nine points, each plane consisting of points $j_i$ with $i \in \FF_p$ and $j$ running over the indices in one of the disjoint $p$-cycles that make up $\tau$.

We let
\begin{equation}
G = A \rtimes B
\end{equation}
and notice that $G$ acts transitively on $\Omega$ (indeed, $\langle \sigma_0,\gamma \rangle$ is already transitive on $\Omega$).
\end{proof}

It would be quite tedious to write explicitly the generators of the group constructed in the proof of Lemma \ref{hberger} for general $p$ and $q$, but we will make this construction more clear by giving an example with $d=21 = 3\cdot7$.  
\begin{example}
\normalfont
We assume the notation of the preceding proof.  The 7-cycle $\sigma = (1~2~3~4~5~6~7)$ is normalized by the 6-cycle $\eta = (2~6~5~7~3~4)$.  Squaring this permutation yields a product of 3-cycles $\tau = (2~5~3)(6~7~4)$, which normalizes $\sigma$.  As described above, we have $$\Omega = \big\{j_i ~\big|~ i \in \FF_p, j \in \{1,\dots,7\}\big\}$$.  The permutations defined in the proof are given as follows:

\begin{align*}
\sigma_0 &=  \big(1_0~2_0~3_0~4_0~5_0~6_0~7_0\big),\\
\sigma_1 &=  \big(1_1~2_1~3_1~4_1~5_1~6_1~7_1\big),\\
\sigma_2 &=  \big(1_2~2_2~3_2~4_2~5_2~6_2~7_2\big),\\
\tau_0 &= \big(2_0~5_0~3_0\big)\cdot \big(6_0~7_0~4_0\big),\\
\tau_1 &= \big(2_1~5_1~3_1\big)\cdot \big(6_1~7_1~4_1\big),\\
\tau_2 &= \big(2_2~5_2~3_2\big)\cdot \big(6_2~7_2~4_2\big),\\
\alpha &= \tau_0\tau_1\tau_2,\\
\beta &= \tau_1 \tau_2^2,\\
\gamma &= 
\big(1_0~1_1~1_2\big)\cdot
\big(2_0~2_1~2_2\big)
\cdots
\big(7_0~7_1~7_2\big),~\textup{ and}\\
G &= \langle \sigma_0, \sigma_1 \sigma_2 \rangle \rtimes \langle \alpha, \beta, \gamma \rangle.
\end{align*}

%
%

To verify that $\langle \alpha, \beta, \gamma \rangle \cong H_3$ as given by (\ref{matrixiso}), we consider the following way of visualizing $\Omega$.


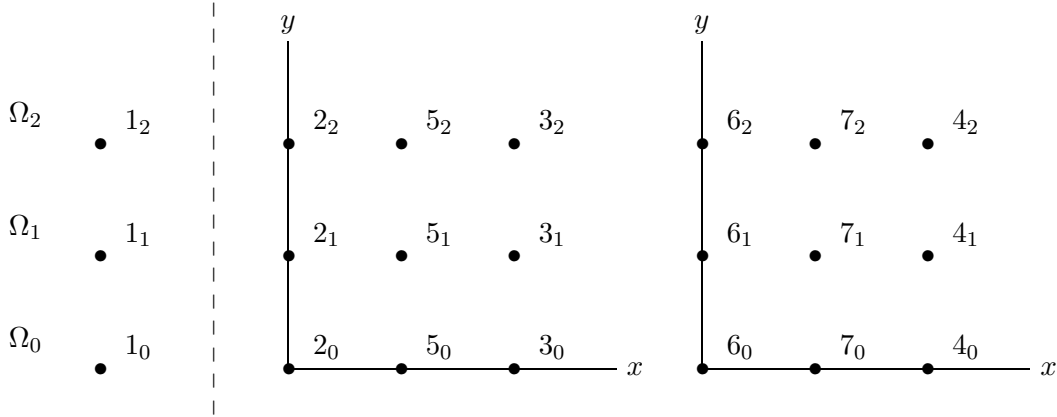
\begin{equation*}
\begin{xy}
(-10,19)*+{\Omega_2}="o1";(-10,4)*+{\Omega_1}="o2";(-10,-11)*+{\Omega_0}="o3";
(15,35)*+{~}="g1";(15,-22)*+{~}="g2";
(24.9,31)*+{y}="y1";(79.9,31)*+{y}="y2";
(24.9,-16)*+{~}="y1b";(79.9,-16)*+{~}="y2b";
(71,-14.9)*+{x}="x1";(126,-14.9)*+{x}="x2";
(69,-14.9)*+{~}="x1a";(124,-14.9)*+{~}="x2a";
(22,-14.9)*+{~}="x1b";(77,-14.9)*+{~}="x2b";
(0,15)*+{\bullet}="11";(25,15)*+{\bullet}="12";(40,15)*+{\bullet}="15";(55,15)*+{\bullet}="13";(80,15)*+{\bullet}="16";(95,15)*+{\bullet}="17";(110,15)*+{\bullet}="14";
(0,0)*+{\bullet}="21";(25,0)*+{\bullet}="22";(40,0)*+{\bullet}="25";(55,0)*+{\bullet}="23";(80,0)*+{\bullet}="26";(95,0)*+{\bullet}="27";(110,0)*+{\bullet}="24";
(0,-15)*+{\bullet}="21";(25,-15)*+{\bullet}="32";(40,-15)*+{\bullet}="35";(55,-15)*+{\bullet}="33";(80,-15)*+{\bullet}="36";(95,-15)*+{\bullet}="37";(110,-15)*+{\bullet}="24";
(5,18)*+{1_2}="11a";(30,18)*+{2_2}="12a";(45,18)*+{5_2}="15a";(60,18)*+{3_2}="13a";(85,18)*+{6_2}="16a";(100,18)*+{7_2}="17a";(115,18)*+{4_2}="14a";
(5,3)*+{1_1}="21a";(30,3)*+{2_1}="22a";(45,3)*+{5_1}="25";(60,3)*+{3_1}="23a";(85,3)*+{6_1}="26a";(100,3)*+{7_1}="27a";(115,3)*+{4_1}="24";
(5,-12)*+{1_0}="31a";(30,-12)*+{2_0}="32a";(45,-12)*+{5_0}="35a";(60,-12)*+{3_0}="33a";(85,-12)*+{6_0}="36a";(100,-12)*+{7_0}="37a";(115,-12)*+{4_0}="34a";
{\ar@{--} "g1";"g2"}
{\ar@{-} "y1";"y1b"}{\ar@{-} "y2";"y2b"}
{\ar@{-} "x1a";"x1b"}{\ar@{-} "x2a";"x2b"}
\end{xy}
\end{equation*}
Shown are two copies of the affine plane $z = 1$ inside of $\FF_3^3 = \{(x,y,z) ~\big|~x,y,z \in \FF_3\}$.  These eighteen points, together with the three points on the left, correspond to elements of $\Omega$ by the labelings.  For example, the point $(2,0,1)$ in the plane on the left corresponds to $3_0\in \Omega_0$.  The blocks $\Omega_i$ are represented as the three horizontal rows in the diagram.  The columns have been partitioned according to the cycle decomposition of permutations $\tau_i$, so that $\alpha$, $\beta$, and $\gamma$ act via the matrices given in (\ref{matrixiso}), simultaneously on each plane of nine points.

\end{example}

\begin{proof}[Proof of Theorem \ref{squareeven}]
Recall that if $c$ divides $d$, then $k^{[c]} \subseteq k^{[d]}$.  Since $H_p$ is solvable and transitive of degree $p^2$, if follows immediately from Lemma \ref{heisenberger} that $k^{[p^2]}$ is Galois unbounded over $k$ for any prime $p$, yielding part (a).  Checcoli showed how to realize these groups explicitly in \cite{checcolithesis}.  Since the group constructed in Lemma \ref{hberger} is solvable, we again apply Lemma \ref{heisenberger} to see that $k^{[pq]}$ is Galois unbounded over $k$ whenever $p$ and $q$ are primes with $q \equiv 1 \imod{p}$.  This gives part (b).

\end{proof}

\begin{proof}[Proof of Theorem \ref{bdd}]
We know that $k^{[2]} = k^{(2)}_{\textup{ab}}$
, so $k^{[2]}/k$ is bounded.  If $d > 2$, then $d$ is divisible by $c$, where $c$ is either 4 or an odd prime.  We have $k^{[c]} \subseteq k^{[d]}$, and by Corollary \ref{punbounded} and part (a) of Theorem \ref{squareeven}, $k^{[c]}$ is unbounded over $k$.
\end{proof}

We remark that our proofs actually demonstrate that $k^{[d]}/k$ is also unbounded in the case where $k$ is a global function field and $d \geq 3$.


\section{Galois boundedness in prime degree}\label{bddsec}
In this section we prove Theorem \ref{galbdd}.
Clearly the general technique for showing boundedness is to find subgroups of small index inside of a Galois group $G$, whose intersection is a given subgroup $H$.  If we want to show Galois boundedness, we take $H$ to be normal.  We will show that we can accomplish this task when $G$ is a subdirect product of transitive groups of prime degree.

The following lemma characterizes the transitive groups of degree $p$.
\begin{lemma}\label{GTB}
If $p$ is a prime number and $G$ is a transitive group of degree $p$, then we have $G = T \rtimes B,$ where $T$ is simple and transitive, and $B$ is a subgroup of $C_{p-1}$.
\end{lemma}
This lemma can be proved by elementary means.  It can also be seen quickly using the classification of finite simple groups: a theorem of Burnside (see \cite{wielandt}, Theorem 11.7; cf. \cite{dixonmortimer}, Theorem 4.1B) implies that $G$ is either a subgroup of $C_p \times C_{p-1}$ containing $C_p$, or an almost simple group, meaning that there is a simple group $T$ such that $T \leq G \leq \Aut(T)$; in this case we also have that $G$ is doubly transitive, meaning that $G$ can send any two points to any other two points.  That $T$ is itself transitive of degree $p$ follows from \cite{wielandt}, Proposition 7.1, which states that every normal subgroup of a primitive permutation group is transitive.  The Classification Theorem for Finite Simple Groups implies that there is a very small list of possibilities for $T$ (see \cite{feit}, Corollary 4.2), and the lemma can be easily checked in these cases.

We are now ready to establish a group theoretic result, of which Theorem \ref{galbdd} will be an immediate corollary.

\begin{proposition}\label{bigboy}
Let $p$ be a prime number and let $G$ be a finite subdirect product of transitive groups of degree $p$.  If $N$ is a normal subgroup of $G$, then $N$ is an intersection of subgroups of index at most $p$ in $G$.  \end{proposition}

\begin{proof}
Let $G \leq_{sd} G_1 \times \cdots \times G_n$, where $G_i$ is a transitive group of degree $p$ for $i \in \oneto{n}$.  If we consider each group $G_i$ acting transitively on a set $\Omega_i$ of size $p$, we have $G$ acting faithfully on the disjoint union of these sets, which we denote by $\Omega$.  Let $\pi_i$ denote the projection onto $G_i$, and let $T_i$ denote the (unique) minimal normal subgroup of $G_i$.  As mentioned following Lemma \ref{GTB}, we know that each $T_i$ is either isomorphic to $C_p$ or to a simple non-abelian group.  We write $K_i = G \cap G_i$, which is a normal subgroup of both $G$ and $G_i$.  We proceed by induction on $n$.  The case $n=1$ follows easily from Lemma \ref{GTB}, since if $N$ is nontrivial we must have $G/N$ abelian of order dividing $p-1$; if $N$ is trivial, observe that the point-stabilizers in $G$ have index $d$ and trivial intersection.

For each $i$ we have that $G/K_i$ is a subdirect product of $\prod_{j \not = i} G_j$.  Notice that we may apply the inductive hypothesis to write $NK_i/K_i$ as an intersection of some subgroups $\{H_l/K_i\}_l$ of index at most $p$ in $G/K_i$.  Now the subgroups $\{H_l\}_l$ are of index at most $p$ in $G$, and $NK_i = \cap_l H_l$.  If $K_i$ is trivial, then our proof is complete.  Alternatively, if $N$ acts trivially on $\Omega_i$, notice that 
\begin{equation}
N = \big(\cap_{x \in \Omega_i} \Stab_G(x) \big) \cap NK_i =\big(\cap_{x \in \Omega_i} \Stab_G(x) \big) \cap \big(\cap_l H_l\big).
\end{equation}
Since the stabilizers $\Stab_G(x)$ have index $p$ in $G$, we have written $N$ as an intersection of subgroups of index at most $p$ in $G$.  Thus we may assume that, for each $i$, the subgroup $K_i$ is nontrivial, and $N$ acts nontrivially on $\Omega_i$.  Moreover, since $K_i$ is nontrivial and normalized by $G_i$, it follows that $T_i \leq K_i \leq G$.  We write $T = \prod_i T_i$ so that $T \leq G$, and $G/T$ is abelian of exponent dividing $p-1$.  

Since $N$ acts nontrivially on each $\Omega_i$, we know that $T_i \leq \pi_i(N)$.  For each $i$ such that $T_i$ is non-abelian (recall that $T_i$ is simple), we will have $T_i = [T_i,N] \leq N$.  Write $T_{\textup{ab}}$ for the product of the $T_i$ which are abelian (these are all isomorphic to $C_p$), and write $T_{\textup{n}}$ for the product of those which are non-abelian.  We have $T_{\textup{n}} \leq N$, so
\begin{equation}
\frac{TN}{N} = \frac{T_{\textup{ab}}T_{\textup{n}}N}{N} = \frac{T_{\textup{ab}}N}{N} \cong \frac{T_{\textup{ab}}}{T_{\textup{ab}}\cap N}.
\end{equation}
Therefore $TN/N$ is an elementary abelian $p$-group.  We also know that $G/TN$ is abelian of exponent dividing $p-1$, so the short exact sequence
\begin{equation}\label{g/n}
1 \to TN/N \to G/N \to G/TN \to 1
\end{equation}
splits by the Schur-Zassenhaus Theorem (Theorem 39 from Chapter 17 of \cite{df}).  Let $V = TN/N$ and $B = G/TN$, so (\ref{g/n}) gives us
\begin{equation}\label{boom}
G/N = V \rtimes B.
\end{equation}

We want to show that there is a collection of subgroups of index at most $p$ in $G/N$ whose intersection is trivial.  It is clear that we can find such subgroups whose intersection is $V$, since $B$ is abelian of exponent dividing $p-1$.  Therefore it suffices to find subgroups of $G/N$ of index at most $p$ whose intersection meets $V$ trivially.

Considering the $\FF_p$- vector space $V$ as a $B$-module, Maschke's Theorem (Theorem 1 from Chapter 18 of \cite{df}) tells us that $V$ decomposes as a direct sum of irreducible $B$-modules.  Since $x^{p-1}-1$ splits over $\FF_p$, it follows that these irreducible submodules are one dimensional.  Now we have submodules $V_i$ of index $p$ (codimension-one submodules), which yield subgroups $V_i \rtimes B$ of index $p$ in $G/N$, and the intersection of all of these meets $V$ trivially. 
\end{proof}

\begin{proof}[Proof of Theorem \ref{galbdd}]
Let $k$ be any field, and let $K/k$ be a finite Galois subextension of $k^{[p]}/k$, where $p$ is prime.  This implies that $K$ is contained a compositum $L$ of the splitting fields of finitely many irreducible, separable polynomials of degree $p$ over $k$.  Let $G = \Gal(L/k)$ and $N = \Gal(L/K)$.  Then $G$ isomorphic to a subdirect product of transitive groups of degree $p$, and $N$ is normal in $G$.  Theorem \ref{bigboy} implies that $N$ is an intersection of subgroups of index at most $p$ in $G$.  By the Galois correspondence, this means that $K$ is the compositum of finitely many extensions of $k$ of degree at most $p$.  Therefore, $K/k$ is generated by elements of degree at most $p$.  (In fact, it must be generated by elements whose degrees are either equal to $p$ or divide $p-1$.)
\end{proof}


\bibliography{bg2}{}
\baselineskip=11pt

\end{document}